\documentclass{amsart}%
\usepackage{amssymb}
\usepackage{amsfonts}%
\usepackage{amsmath}%
\usepackage{graphicx}
\providecommand{\U}[1]{\protect\rule{.1in}{.1in}}
\newtheorem{theorem}{Theorem}
\theoremstyle{plain}

\newtheorem{corollary}{Corollary}

\newtheorem{definition}{Definition}
\newtheorem{example}{Example}

\newtheorem{proposition}{Proposition}
\newtheorem{remark}{Remark}

\numberwithin{equation}{section}
\begin{document}
\title[On weakly 1-absorbing primary ideals of commutative rings]{On weakly 1-absorbing primary ideals of commutative rings}
\author{Ayman Badawi}
\address{Department of Mathematics \&Statistics, The American University of Sharjah,
P.O. Box 26666, Sharjah, United Arab Emirates.}
\email{abadawi@aus.edu}
\author{Ece Yetkin Celikel}
\address{Department of Civil Engineering, Faculty of Engineering, Hasan Kalyoncu
University, Gaziantep, Turkey.}
\email{yetkinece@gmail.com}
\date{September, 2019}

\begin{abstract}
Let R be a commutative ring with $1\neq0$. In this paper, we introduce the
concept of weakly 1-absorbing primary ideal which is a generalization of
1-absorbing ideal. A proper ideal $I$ of $R$ is called a \textit{weakly
1-absorbing primary} ideal if whenever nonunit elements $a,b,c\in R$ and
$0\neq abc\in I,$ then $ab\in I$ or $c\in\sqrt{I}$. A number of results
concerning weakly 1-absorbing primary ideals and examples of weakly
1-absorbing primary ideals are given. Furthermore, we give the correct version
of a result on 1-absorbing ideals of commutative rings.

\end{abstract}
\maketitle

\section{Introduction}

Throughout this paper, all rings are commutative with nonzero identity. Let
$R$ be a commutative ring. By a proper ideal $I$ of $R$, we mean an ideal $I$
of $R$ with $I\neq R$. Let $I$ be a proper ideal of $R$. Before we state some
results, let us introduce some notation and terminology. By $\sqrt{I},$ we
mean the radical of $R$, that is, $\{a\in R \mid a^{n}\in I$ for some positive
integer $n\}$. In particular, $\sqrt{0}$ denotes the set of all nilpotent
elements of $R$. We define $Z_{I}(R) = \{r \in R \mid rs \in I$ for some $s
\in R\setminus I\}$. A ring $R$ is called a reduced ring if it has no non-zero
nilpotent elements; i.e., $\sqrt{0}=0$. For two ideals $I$ and $J $ of $R,$
the residual division of $I$ and $J$ is defined to be \ the ideal
$(I:J)=\{a\in R \mid aJ\subseteq I\}$. Let $R$ be a commutative ring with
identity and $M$ a unitary $R$-module. Then $R(+)M=R\oplus M$ (direct sum)
with coordinate-wise addition and multiplication $(a,m)(b,n)=(ab,an+bm)$ is a
commutative ring with identity called the idealization of $M$. A ring $R $ is
called a \textit{quasilocal} ring if $R$ has exactly one maximal ideal. As
usual we denote $%
\mathbb{Z}
$ and $%
\mathbb{Z}
_{n}$ by the ring of integers and the ring of integers modulo $n$.

Since prime and primary ideals have key roles in commutative ring theory, many
authors have studied generalizations of prime and primary ideals. Anderson and
Smith introduced in \cite{And} the notion of weakly prime ideals. A proper
ideal $I$ of $R$ is called a \textit{weakly prime} ideal of $R$ if whenever
$a,b\in R$ and $0\neq ab\in I$, then $a\in I$ or $b\in I$. Then Atani and
Farzalipour introduced the concept of weakly primary ideals which is a
generalization of primary ideals in \cite{At}.\ A proper ideal $I$ of $R$ is
called a \textit{weakly primary} ideal of $R$ if whenever $a,b\in R $ and
$0\neq ab\in I$, then $a\in I$ or $b\in\sqrt{I}$. For a different
generalizations of prime ideals and weakly prime ideals, the contexts of
2-absorbing and weakly 2-absorbing ideals were defined. According to
\cite{Badawi} and \cite{BadDar}, a proper ideal $I$ of $R$ is called a
\textit{2-absorbing} (\textit{weakly 2-absorbing}) ideal of $R$, if whenever
$a,b,c\in I$ and $abc\in I$ ($0\neq abc\in I$)$,$ then $ab\in I$ or $bc\in I$
or $ac\in I.$ As a generalization of 2-absorbing and weakly 2-absorbing
ideals, 2-absorbing primary and weakly 2-absorbing primary ideals were defined
in \cite{Badawi2} and \cite{Badawi3}, respectively. A proper ideal $I $ of $R$
is said to be \textit{2-absorbing primary} (\textit{weakly 2-absorbing
primary})\ if whenever $a,b,c\in R$ and $abc\in I$ ($0\neq abc\in I$)$,$ then
$ab\in I$ or $bc\in\sqrt{I}$ or $ac\in\sqrt{I}.$ In a recent study
\cite{BadEce}, we call a proper ideal $I$ of $R$ a \textit{1-absorbing
primary} ideal if whenever nonunit elements $a,b,c$ $\in R$ and $abc\in I$,
then $ab\in I$ or $c\in\sqrt{I}.$

In this paper, we introduce the concept of weakly 1-absorbing ideal of a ring
$R$. A proper ideal $I$ of $R$ is called a \textit{weakly 1-absorbing primary}
ideal of $R$ if whenever nonunit elements $a,b,c$ $\in R$ and $0 \not = abc\in
I$, then $ab\in I$ or $c\in\sqrt{I}.$ It is clear that a 1-absorbing primary
ideal of $R$ is a weakly 1-absorbing primary ideal of $R$. However, since $0$
is always weakly 1-absorbing primary, a weakly 1-absorbing primary ideal of
$R$ needs not be a 1-absorbing primary ideal of $R$ (see Example \ref{e1}).

Among many results, we show (Theorem \ref{MAX}) that if a proper ideal $I$ of
$R$ is a weakly 1-absorbing ideal of $R$ such that $\sqrt{I}$ is a maximal
ideal of $R$, then $I$ is a primary ideal of $R$, and hence $I$ is 1-absorbing
primary ideal of $R$. We show (Theorem \ref{rad(I)}) that If $R$ is a reduced
ring and $I$ is a weakly 1-absorbing primary ideal of $R$, then that $\sqrt
{I}$ is a prime ideal of $R$. If $I$ is a proper nonzero ideal of a
von-Neumann regular ring $R$, then we show (Theorem \ref{VNI}) that $I$ is a
weakly 1-absorbing primary ideal of $R$ if and only if $I$ is a 1-absorbing
primary ideal of $R$ if and only if $I$ is a primary ideal of $R$. We show
(Theorem \ref{NQ}) that if $R$ be a non-quasilocal ring and $I$ be a proper
ideal of $R$ such that $ann(i) = \{r \in R\mid ri = 0\}$ is not a maximal
ideal of $R$ for every element $i \in I$, then $I$ is a weakly 1-absorbing
primary ideal of $R$ if and only if $I$ is a weakly primary ideal of $R$. If
$I$ is a proper ideal of a reduced divided ring $R$, then we show (Theorem
\ref{d}) that $I$ is a weakly 1-absorbing primary ideal of $R$ if and only if
$I$ is a weakly primary ideal of $R$. If $I$ is a weakly 1-absorbing primary
of a ring $R$ that is not a 1-absorbing primary ideal of $R$, then we give
(Theorem \ref{abI}) sufficient conditions so that $I^{3} = 0$ (i.e., $I
\subseteq\sqrt{I}$). In Theorem \ref{ch}, we obtain some equivalent conditions
for weakly 1-absorbing primary ideals of $u$-rings. We give (Theorem \ref{w1})
a characterization of weakly 1-absorbing primary ideals in $R=R_{1}\times
R_{2}$ where $R_{1}$ and $R_{2}$ are commutative rings with identity that are
not fields. If $R_{1},R_{2},..., R_{n}$ are commutative rings with identity
for some $2 \leq n < \infty$ and $R=R_{1}\times R_{2}\times\cdots\ R_{n}$,
then it is shown (Theorem \ref{fi}) that every proper ideal of $R$ is a weakly
1-absorbing primary ideal of $R, $ if and only if $n = 2$ and $R_{1}, R_{2}$
are fields. For a weakly 1-absorbing primary ideal of a ring $R$, we show
(Theorem \ref{S}) that $S^{-1}I$ \ is a weakly 1-absorbing primary ideal of
$S^{-1}R$ for every multiplicatively closed subset $S$ of $R$ that is disjoint
from $I$, and we show that the converse holds if $S\cap Z(R)=S\cap
Z_{I}(R)=\emptyset.$ We give (Remark \ref{TT}) the correct versions of
\cite[Theorem 17(1), Corrollary 3 and Corollary 4]{BadEce}.

\section{Properties of Weakly 1-absorbing primary ideals}

\begin{definition}
Let $R$ be a commutative ring and $I$ a proper ideal of $R$. We call $I$ a
weakly 1-absorbing primary ideal of $R$ if whenever nonunit elements $a,b,c$
$\in R$ and $0\neq abc\in I$, then $ab\in I$ or $c\in\sqrt{I}.$
\end{definition}

It is clear that every 1-absorbing primary ideal of a ring $R$ is a weakly
1-absorbing primary ideal of $R$, and $I=\{0\}$ is a weakly 1-absorbing
primary ideal of $R.$ The following example shows that the converse is not true.

\begin{example}
\begin{enumerate}
\item \label{e1}$I=\{0\}$ is a weakly 1-absorbing primary ideal of $R =%
\mathbb{Z}
_{6}$ that is not a 1-absorbing primary of $R$. Indeed, $2\cdot2\cdot3\in I $
but neither $2\cdot2\in I$ nor $3\in\sqrt{I}.$

\item Let $J=\{0,6\}$ as an ideal of $%
\mathbb{Z}
_{12}$ and let $R=%
\mathbb{Z}
_{12}(+)J$. Then an ideal $I=\{(0,0),(0,6)\}$ is a weakly 1-absorbing primary
ideal of $R.$ Observe that $abc\in I$ for some $a,b,c\in R\backslash I$ if and
only if $abc=(0,0).$ However it is not a 1-absorbing primary ideal of $R$.
Indeed; $(2,0)(2,0)(3,0)\in I$, but neither $(2,0)(2,0)\in I$ nor
$(3,0)\in\sqrt{I}$.

\item For an infinite example of a weakly 1-absorbing primary ideal that is
not 1-absorbing primary, put $J=\{0,6\}$ an ideal of $%
\mathbb{Z}
_{12},\ $and consider the idealization ring $R=%
\mathbb{Z}
_{12}(+)J[X]$ and the ideal $I=\{0\}(+)J[X]$. Then $I$ is an infinite ideal of
$R.$ Since $abc\in I$ for some $a,b,c\in R\backslash I$ if and only if
$abc=(0,0)$, then $I$ is a weakly 1-absorbing primary ideal of $R$.
\end{enumerate}
\end{example}

We begin with the following trivial result without proof.

\begin{theorem}
\label{tr}Let $I$ be a proper ideal of a commutative ring $R$. Then the
following statements hold.

\begin{enumerate}
\item If $I$ is a weakly prime ideal, then $I$ is a weakly 1-absorbing primary ideal.

\item If $I$ is a weakly primary ideal, then $I$ is a weakly 1-absorbing
primary ideal.

\item If $I$ is a 1-absorbing primary ideal, then $I$ is a weakly 1-absorbing
primary ideal.

\item If $I$ is a weakly 1-absorbing primary ideal, then $I$ is a weakly
2-absorbing primary ideal.

\item If $R/$ is an integral domain, then $I$ is a weakly 1-absorbing primary
ideal if and only if $I$ is a 1-absorbing primary ideal of $R.$

\item Let $R$ be a quasilocal ring with maximal ideal $\sqrt{0}.$ Then every
proper ideal of $R$ is a weakly 1-absorbing primary ideal of $R$.
\end{enumerate}
\end{theorem}

We recall that a proper ideal $I$ of $R$ is called a \textit{semiprimary}
ideal of $R$ if $\sqrt{I}$ is a prime ideal of $R$. For an interesting article
on semiprimary ideals of commutative rings see \cite{GilmerR}. For a recent
related article on semiprimary ideals, we recommend \cite{BadDenis}. We have
the following result.

\begin{theorem}
\label{MAX} Let $R$ be a ring and $I$ be a weakly 1-absorbing primary ideal of
$R$. If $\sqrt{I}$ is a maximal ideal of $R$, then $I$ is a primary ideal of
$R$, and hence $I$ is a 1-absorbing ideal primary of $R$. In particular, If
$I$ a weakly 1-absorbing primary ideal of $R$ that is not a 1-absorbing ideal
primary of $R$, then $\sqrt{I}$ is not a maximal ideal of $R$.
\end{theorem}

\begin{proof}
Suppose that $\sqrt{I}$ is a maximal ideal of $R$. Then $I$ is a semiprimary
ideal of $R$. Since $I$ is a semiprimary ideal of $R$ and $\sqrt{I}$ is a
maximal ideal of $R$, we conclude that $I$ is a primary ideal of $R$ by
\cite[P. 153]{OP}. Thus $I$ is a 1-absorbing primary ideal of $R$.
\end{proof}

\begin{theorem}
\label{rad(I)}Let $R$ be a reduced ring. If $I$ is a nonzero weakly
1-absorbing primary ideal of $R$, then $\sqrt{I}$ is a prime ideal of $R$. In
particular, if $\sqrt{I}$ is a maximal ideal of $R$, then $I$ is a primary
ideal of $R$, and hence $I$ is a 1-absorbing primary ideal of $R$.
\end{theorem}

\begin{proof}
Suppose that $0\neq ab\in\sqrt{I}$ for some $a,b\in R$. We may assume that
$a,b$ are nonunit. Then there exists an even positive integer $n=2m$
$(m\geq1)$ such that $(ab)^{n}\in I$. Since $\sqrt{0} = \{0\}$, we have
$(ab)^{n}\neq0$. Hence $0\neq a^{m}a^{m}b^{n}\in I$. Thus $a^{m}a^{m}=a^{n}\in
I$ or $b^{n}\in\sqrt{I}$, and therefore $\sqrt{I}$ is a weakly prime ideal of
$R$. Since $R$ is reduced and $I \not = \{0\}$, we conclude that $\sqrt{I}$ is
a prime ideal of $R$ by \cite[Corollary 2]{And}. The proof of the "in
particular" statement is now clear by Theorem \ref{MAX}.
\end{proof}

Recall that a commutative ring $R$ is called a \textit{von-Neumann regular}
ring if and only if for every $x \in R$, there is a $y \in R$ such that
$x^{2}y = x$. It is known that a commutative ring $R$ is a von-Neumann regular
ring if and only if for each $x \in R$, there is an idempotent $e \in R$ and a
unit $u \in R$ such that $x = eu$. For a recent article on von-Neumann regular
rings see\cite{AndBad2}. We have the following result.

\begin{theorem}
\label{VNI} Let $R$ be a von-Neumann regular ring and $I$ be a nonzero ideal
of $R$. Then the following statements are equivalent.

\begin{enumerate}
\item $I$ is a weakly 1-absorbing primary ideal of $R$.

\item $I$ is a primary ideal of $R$.

\item $I$ is a 1-absorbing ideal primary of $R$.
\end{enumerate}
\end{theorem}

\begin{proof}
(1)$\rightarrow$(2). Since $R$ is a von-Neumann regular ring, we know that $R
$ is reduced. Hence $\sqrt{I}$ is a prime ideal of $R$ by Theorem
\ref{rad(I)}. Since every prime ideal of a von-Neumann regular ring is
maximal, we conclude that $\sqrt{I}$ is a maximal ideal of $R$. Hence $I$ is a
primary ideal of $R$ by Theorem \ref{MAX}.

(2)$\rightarrow$(3)$\rightarrow$(1). It is clear.
\end{proof}

\begin{theorem}
\label{NQ} Let $R$ be a non-quasilocal ring and $I$ be a proper ideal of $R$
such that $ann(i) = \{r \in R\mid ri = 0\}$ is not a maximal ideal of $R$ for
every element $i \in I$. Then $I$ is a weakly 1-absorbing primary ideal of $R$
if and only if $I$ is a weakly primary ideal of $R$.
\end{theorem}

\begin{proof}
If $I$ is a weakly primary ideal of $R$, then $I$ is a weakly 1-absorbing
primary ideal of $R$ by Theorem \ref{tr}(2). Hence suppose that $I$ is a
weakly 1-absorbing primary ideal of $R$ and suppose that $0 \not = ab \in$ for
some elements $a, b \in R$. We show that $a \in I$ or $b \in\sqrt{I}$. We may
assume that $a, b$ are nonunit elements of $R$. Let $ann(ab) = \{c \in R\mid
cab = 0\}$. Since $ab \not = 0$, $ann(ab)$ is a proper ideal of $R$. Let $L$
be a maximal ideal of $R$ such that $ann(ab) \subset L$. Since $R$ is a
non-quasilocal ring, there is a maximal ideal $M$ of $R$ such that $M \not =
L$. Let $m \in M\setminus L$. Hence $m \notin ann(ab)$ and $0 \not = mab \in
I$. Since $I$ is a weakly 1-absorbing primary ideal of $R$, we have $ma \in I$
or $b\in\sqrt{I}$. If $b \in\sqrt{I}$, then we are done. Hence assume that $b
\notin\sqrt{I}$. Hence $ma \in I$. Since $m \notin L$ and $L$ is a maximal
ideal of $R$, we conclude that $m \notin J(R)$. Hence there exists an $r\in R$
such that $1 + rm$ is a nonunit element of $R$. Suppose that $1 + rm \notin
ann(ab)$. Hence $0 \not = (1+rm)ab \in I$. Since $I$ is a weakly 1-absorbing
primary ideal of $R$ and $b\notin\sqrt{I}$, we conclude that $(1 + rm)a = a +
rma \in I$. Since $rma \in I$, we have $a \in I$ and we are done. Suppose that
$1 + rm \in ann(ab)$. Since $ann(ab)$ is not a maximal ideal of $R$ and
$ann(ab) \subset L$, there is a $w \in L \setminus ann(ab)$. Hence $0 \not =
wab \in I$. Since $I$ is a weakly 1-absorbing primary ideal of $R$ and $b
\not \in \sqrt{I}$, we conclude that $wa \in I$. Since $1 + rm \in ann(ab)
\subset L$ and $w \in L \setminus ann(ab)$, we have $1 + rm + w$ is a nonzero
nonunit element of $L$. Hence $0 \not = (1 + rm + w)ab \in I$. Since $I$ is a
weakly 1-absorbing primary ideal of $R$ and $b \not \in \sqrt{I}$, we conclude
that $(1 + rm + w)a = a + rma + wa \in I$. Since $rma, wa \in I$, we conclude
that $a\in I$.
\end{proof}

\textbf{Question}. Is Theorem \ref{NQ} still valid without the assumption that
$ann(i) = \{r \in R\mid ri = 0\}$ is not a maximal ideal of $R$ for every
element $i \in I$? We are unable to give a proof of Theorem \ref{NQ} without
this assumption.

In light of the proof of Theorem \ref{NQ}, we have the following result.

\begin{theorem}
\label{NOUNIT} Let $I$ be a weakly 1-absorbing primary ideal of $R$ such that
for every nonzero element $i \in I$, there exists a nonunit $w \in R$ such
that $wi \not = 0$ and $w + u$ is a nonunit element of $R$ for some unit $u
\in R$. Then $I$ is a weakly primary ideal of $R$.
\end{theorem}

\begin{proof}
Suppose that $0 \not = ab \in I$ and $b \notin\sqrt{I}$ for some $a, b \in R
$. We may assume that $a, b$ are nonunit elements of $R$. Hence there is a
nonunit $w \in R$ such that $wab \not = 0$ and $w + u$ is a nonunit element of
$R$ for some unit $u \in R$. Since $0 \not = wab \in I$ and $b \notin\sqrt{I}$
and $I$ is a weakly 1-absorbing primary ideal of $R$, we conclude that $wa \in
I$. Since $(w + u)ab \in I$ and $I$ is a weakly 1-absorbing primary ideal of
$R$ and $b \notin\sqrt{I}$, we conclude that $(w + u)a = wa + ua \in I$. Since
$wa \in I$ and $wa + ua \in I$, we conclude that $ua \in I$. Since $u$ is a
unit, we have $a \in I$.
\end{proof}

\begin{corollary}
Let $R$ be a ring and $A = R[X]$. Suppose that $I$ is a weakly 1-absorbing
primary ideal of $A$. Then $I$ is a weakly primary ideal of $A$.
\end{corollary}

\begin{proof}
Since $Xi \not = 0$ for every nonzero $i \in I$ and $X + 1$ is a nonunit
element of $A$, we are done by Theorem \ref{NOUNIT}.
\end{proof}

Recall that a ring $R$ is called \textit{divided} if for every prime ideal $P
$ of $R$ and for every $x\in R\setminus P$, we have $x\mid p$ for every $p\in
P$. We have the following result.

\begin{theorem}
\label{d}Let $R$ be a reduced divided ring and $I$ be a proper ideal of $R.$
Then the following statements are equivalent:
\end{theorem}

\begin{enumerate}
\item $I$ is a weakly 1-absorbing primary ideal of $R$.

\item $I$ is a weakly primary ideal of $R$.
\end{enumerate}

\begin{proof}
(1)$\Rightarrow$(2). Suppose that $0\neq ab\in I$ for some $a,b\in R$ and
$b\notin\sqrt{I}$. We may assume that $a,b$ are nonunit elements of $R$. Since
$\sqrt{I}$ is a prime ideal of $R$ by Theorem \ref{rad(I)}, we conclude that
$a\in\sqrt{I}$. Since $R$ is divided, we conclude that $b\mid a$. Thus $a=bc$
for some $c\in R$. Observe that $c$ is a nonunit element of $R$ as
$b\notin\sqrt{I}$ and $a\in\sqrt{I}$. Since $0\neq ab=bcb\in I$ and $I$ is
weakly 1-absorbing primary, and $b\notin\sqrt{I}$, we conclude that $bc=a\in
I$. Thus $I$ is a weakly primary ideal of $R$. (2)$\Rightarrow$(1). It is
clear by Theorem \ref{tr}(2).
\end{proof}

Recall that a ring $R$ is called a \textit{chained} ring if for every $x,y\in
R$, we have $x\mid y$ or $y\mid x$. Every chained ring is divided. So, if $R$
is a reduced chained ring, then a proper ideal $I$ of $R $ is a weakly
1-absorbing primary ideal if and only if it is a weakly primary ideal of $R$.

\begin{theorem}
\label{Ded} Let $R$ be a Dedekind domain and $I$ be a nonzero proper ideal of
$R$. Then $I$ is a weakly 1-absorbing primary ideal of $R$ if and only if
$\sqrt{I}$ is a prime ideal of $R$.
\end{theorem}

\begin{proof}
Suppose that $I$ is a weakly 1-absorbing primary ideal of $R$. Then $\sqrt{I}
$ is a prime ideal of $R$ by Theorem \ref{rad(I)}. The converse part follows
from \cite[Theorem 14]{BadEce}.
\end{proof}

Let $R$ be a commutative ring. If an ideal of $R$ contained in a finite union
of ideals must be contained in one of those ideals, then $R$ is said to be a
\textit{$u$-ring} \cite{Qu}. In the next theorem, we give some
characterizations of weakly 1-absorbing primary ideals in $u$-rings.

\begin{theorem}
\label{ch}Let $R$ be a commutative u-ring, and $I$ a proper ideal of $R$. Then
the following statements are equivalent.
\end{theorem}

\begin{enumerate}
\item $I$ is a weakly 1-absorbing primary ideal of $R.$

\item For every nonunit elements $a,b\in R$ with $ab\notin I$, $(I:ab)=(0:ab)
$ or $(I:ab)\subseteq\sqrt{I}.$

\item For every nonunit element $a\in R$ and every ideal $I_{1}$ of $R$ with
$I_{1}\nsubseteq\sqrt{I}$, if $(I:aI_{1})$ is a proper ideal of $R$, then
$(I:aI_{1})=(0:aI_{1})$ or $(I:aI_{1})\subseteq(I:a).$

\item For every ideals $I_{1},I_{2}$ of $R$ with $I_{1}\nsubseteq\sqrt{I},$ if
$(I:I_{1}I_{2})$ is a proper ideal of $R,$ then $(I:I_{1}I_{2})=(0:I_{1}%
I_{2})$ or $(I:I_{1}I_{2})\subseteq(I:I_{2}).$

\item For every ideals $I_{1},I_{2},I_{3}$ of $R$ with $0\neq I_{1}I_{2}%
I_{3}\subseteq I$, $I_{1}I_{2}\subseteq I$ or $I_{3}\subseteq\sqrt{I}.$
\end{enumerate}

\begin{proof}
(1)$\Rightarrow$(2) Suppose that $I$ is a weakly 1-absorbing primary ideal of
$R$, $ab\notin I$ for some nonunit elements $a,b\in R$ and $c\in(I:ab)$. Then
$abc\in I$. Since $ab\notin I,$ $c$ is nonunit. If $abc=0$, then $c\in(0:ab)$.
Assume that $0\neq abc\in I$. Since $I$ is weakly 1-absorbing primary, we have
$c\in\sqrt{I}.$ Hence we conclude that $(I:ab)\subseteq(0:ab)\cup\sqrt{I}.$
Since $R$ is a u-ring, we obtain that $(I:ab)=(0:ab)$ or $(I:ab)\subseteq
\sqrt{I}.$

(2)$\Rightarrow$(3) If $aI_{1} \subseteq I$, then we are done. Suppose that
$aI_{1}\nsubseteq I$ for some nonunit element $a\in R$ and $c\in(I:aI_{1})$.
It is clear that $c$ is nonunit. Then $acI_{1}\subseteq I$. Now $I_{1}%
\subseteq(I:ac).$ If $ac\in I$, then $c\in(I:a).$ Suppose that $ac\notin I$.
Hence $(I:ac)=(0:ac)$ or $(I:ac)\subseteq\sqrt{I}$ by (2). Thus $I_{1}%
\subseteq(0:ac)$ or $I_{1}\subseteq\sqrt{I}$. Since $I_{1}\nsubseteq\sqrt{I}$
by hypothesis, we conclude $I_{1}\subseteq(0:ac);$ i.e. $c\in(0:aI_{1})$. Thus
$(I:aI_{1})\subseteq(0:aI_{1})$ $\cup(I:a).$ Since $R$ is a u-ring, we have
$(I:aI_{1})=(0:aI_{1})$ or $(I:aI_{1})\subseteq(I:a).$

(3)$\Rightarrow$(4) If $I_{1} \subseteq\sqrt{I}$, then we are done. Suppose
that $I_{1}\nsubseteq\sqrt{I}$ and $c\in(I:I_{1}I_{2})$. Then $I_{2}%
\subseteq(I:cI_{1}).$ Since $(I:I_{1}I_{2})$ is proper, $c$ is nonunit. Hence
$I_{2}\subseteq(0:cI_{1})$ or $I_{2}\subseteq(I:c)$ by (3). If $I_{2}%
\subseteq(0:cI_{1})$, then $c\in(I:I_{1}I_{2})$. If $I_{2}\subseteq(I:c),$
then $c\in(I:I_{2})$. So, $(I:I_{1}I_{2})\subseteq(0:I_{1}I_{2})\cup(I:I_{2})$
which implies that $(I:I_{1}I_{2})=(0:I_{1}I_{2})$ or $(I:I_{1}I_{2}%
)\subseteq(I:I_{2}),$ as needed.

(4)$\Rightarrow$(5) It is clear.

(5)$\Rightarrow$(1) Let $a,b,c\in R$ nonunit elements and $0\neq abc\in I$.
Put $I_{1} = aR$, $I_{2} = bR$, and $I_{3} = cR$. Then (1) is now clear by (5).
\end{proof}

\begin{definition}
Let $I$ be a weakly 1-absorbing primary ideal of $R$ and $a,b,c$ be nonunit
elements of $R$. We call $(a,b,c)$ a 1-triple-zero of $I$ if $abc=0,$
$ab\notin I$, and $c\notin\sqrt{I}$.
\end{definition}

Observe that if $I$ is a weakly 1-absorbing primary ideal of $R$ that is not
1-absorbing primary, then there exists a 1-triple-zero $(a,b,c)$ of $I$ for
some nonunit elements $a,b,c\in R.$

\begin{theorem}
\label{abI}Let $I$ be a weakly 1-absorbing primary ideal of $R$, and $(a,b,c)
$ be a $1$-triple-zero of $I.$ Then
\end{theorem}

(1) $abI=0$.

(2) If $a,b\notin(I:c)$, then $bcI=acI=aI^{2}=bI^{2}=cI^{2}=0.$

(3) If $a,b\notin(I:c)$, then $I^{3}=0.$

\begin{proof}
(1) Suppose that $abI\neq0.$ Then $abx\neq0$ for some nonunit $x\in I.$ Hence
$0\neq ab(c+x)\in I.$ Since $ab\notin I,$ $(c+x)$ is nonunit element of $R$.
Since $I$ is a weakly 1-absorbing primary ideal of $R$ and $ab \notin I$, we
conclude that $(c+x)\in\sqrt{I} $. Since $x \in I$, we have $c \in\sqrt{I}$, a
contradiction. Thus $abI=0.$

(2) Suppose that $bcI\neq0$. Then $bcy\neq0$ for some nonunit element $y\in
I.$ Hence $0\neq bcy = b(a+y)c\in I.$ Since $b\notin(I:c)$, we conclude that
$a+y$ is a nonunit element of $R$. Since $I$ is a weakly 1-absorbing primary
ideal of $R$ and $ab \notin I$ and $by \in I$, we conclude that $b(a +
y)\notin I$, and hence $c \in\sqrt{I}$, a contradiction. Thus $bcI=0.$ We show
that $acI=0$. Suppose that $acI\neq0$. Then $acy\neq0$ for some nonunit
element $y\in I.$ Hence $0\neq acy = a(b+y)c\in I.$ Since $a\notin(I:c)$, we
conclude that $b+y$ is a nonunit element of $R$. Since $I$ is a weakly
1-absorbing primary ideal of $R$ and $ab \notin I$ and $ay \in I$, we conclude
that $a(b + y)\notin I$, and hence $c \in\sqrt{I}$, a contradiction. Thus
$acI=0.$ Now we prove that $aI^{2}=0$. Suppose that $axy\neq0$ for some
$x,y\in I$. Since $abI=0$ by (1) and $acI = 0$ by (2), $0\neq
axy=a(b+x)(c+y)\in I$. Since $ab \notin I$, we conclude that $c + y$ is a
nonunit element of $R$. Since $a\notin(I:c)$, we conclude that $b+x$ is a
nonunit element of $R$. Since $I$ is a weakly 1-absorbing primary ideal of
$R$, we have $a(b+x)\in I$ or $(c+y)\in\sqrt{I}.$ Since $x, y \in I$, we
conclude that $ab\in I$ or $c\in\sqrt{I}$, a contradiction. Thus $aI^{2}=0.$
We show $bI^{2} = 0$. Suppose that $bxy\neq0$ for some $x,y\in I$. Since
$abI=0$ by (1) and $bcI = 0$ by (2), $0\neq bxy=b(a+x)(c+y)\in I$. Since $ab
\notin I$, we conclude that $c + y$ is a nonunit element of $R$. Since
$b\notin(I:c)$, we conclude that $a+x$ is a nonunit element of $R$. Since $I$
is a weakly 1-absorbing primary ideal of $R $, we have $b(a+x)\in I$ or
$(c+y)\in\sqrt{I}.$ Since $x, y \in I$, we conclude that $ab\in I$ or
$c\in\sqrt{I}$, a contradiction. Thus $bI^{2}=0.$ We show $cI^{2} = 0$.
Suppose that $cxy\neq0$ for some $x,y\in I$. Since $acI = bcI = 0$ by (2),
$0\neq cxy= (a+x)(b+y)c\in I$. Since $a, b \notin(I:c)$, we conclude that $a +
x$ and b + y are nonunit elements of $R$. Since $I$ is a weakly 1-absorbing
primary ideal of $R$, we have $(a+x)(b+y)\in I$ or $c \in\sqrt{I}.$ Since $x,
y \in I$, we conclude that $ab\in I$ or $c\in\sqrt{I}$, a contradiction. Thus
$cI^{2}=0.$

(3)\ Assume that $xyz\neq0$ for some $x,y,z\in I$. Then $0\neq
xyz=(a+x)(b+y)(c+z)\in I$ by (1) and (2). Since $ab \notin I$, we conclude
$c+z$ is a nonunit element of $R$. Since $a, b \notin(I:c)$, we conclude that
$a+x$ and $b+y$ are nonunit elements of $R$. Since $I$ is a weakly 1-absorbing
primary ideal of $R$, we have $(a+x)(b+y) \in I$ or $c+z \in\sqrt{I}.$ Since
$x, y, z \in I$, we conclude that $ab \in I$ or $c \in\sqrt{I}$, a
contradiction. Thus $I^{3}=0$.
\end{proof}

\begin{theorem}
\label{reduced}

\begin{enumerate}
\item Let $I$ be a weakly 1-absorbing primary ideal of a reduced ring $R$.
Suppose that $I$ is not a 1-absorbing ideal primary ideal of $R$ and $(a, b,
c)$ is a 1-triple-zero of $I$ such that $a, b \notin(I:c)$. Then $I = 0$.

\item Let $I$ be a nonzero weakly 1-absorbing primary ideal of a reduced ring
$R$. Suppose that $I$ is not a 1-absorbing ideal primary ideal of $R$ and $(a,
b, c)$ is a 1-triple-zero of $I$. Then $ac \in I$ or $bc \in I$.
\end{enumerate}
\end{theorem}

\begin{proof}
(1) Since $a, b \in(I:c)$, then $I^{3} = 0$ by Theorem \ref{abI}(3). Since $R$
is reduced, we conclude that $I = 0$.

(2) Suppose that neither $ac \in I$ nor $bc = 0$. Then $I = 0$ by (1), a
contradiction since $I$ is a nonzero ideal of $R$ by hypothesis. Hence if $(a,
b, c)$ is a 1-triple-zero of $I$, then $ac \in I$ or $bc \in I$.
\end{proof}

\begin{theorem}
Let $I$ be a weakly 1-absorbing primary ideal of $R.$ If $I$ is not a weakly
primary ideal of $R$, then there exist an irreducible element $x\in R$ and a
nonunit element $y\in R$ such that $xy\in I$, but neither $x\in I$ nor
$y\in\sqrt{I}$. Furthermore, if $ab\in I$ for some nonunit elements $a,b\in R$
such that neither $a\in I$ nor $b\in\sqrt{I}$, then $a$ is an irreducible
element of $R$.
\end{theorem}

\begin{proof}
Suppose that $I$ is not a weakly primary ideal of $R$. Then there exist
nonunit elements $x,y\in R$ such that $0\neq xy\in I$ with $x\notin I$ ,
$y\notin\sqrt{I}$. Suppose that $x$ is not an irreducible element of $R$. Then
$x=cd$ for some nonunit elements $c,d\in R$. Since $0\neq xy=cdy\in I$ and $I$
is weakly 1-absorbing primary and $y\notin\sqrt{I}$, we conclude that $cd=x\in
I$, a contradiction. Hence $x$ is an irreducible element of $R$.
\end{proof}

In general, the intersection of a family of weakly 1-absorbing primary ideals
need not be a weakly 1-absorbing primary ideal. Indeed, consider the ring $R=%
\mathbb{Z}
_{6}$. Then $I=(2)$ and $J=(3)$ are clearly weakly 1-absorbing primary ideals
of $%
\mathbb{Z}
_{6}$ but $I\cap J=\{0\}$ is not a weakly 1-absorbing primary ideal of $R$ by
Example \ref{e1}. However, we have the following result.

\begin{proposition}
Let $\{I_{i}:i\in\Lambda\}$ be a collection of weakly 1-absorbing primary
ideals of $R$ such that $Q = \sqrt{I_{i}} = \sqrt{I_{j}}$ for every distinct
$i, j \in\Lambda$. Then $I=\cap_{i\in\Lambda}I_{i}$ is a weakly 1-absorbing
primary ideal of $R$.
\end{proposition}

\begin{proof}
Suppose that $0\neq abc\in I=\cap_{i\in\Lambda}I_{i}$ for nonunit elements
$a,b,c$ of $R$ and $ab\notin I$. Then for some $k\in\Lambda$, $0\neq abc\in
I_{k}$ and $ab\notin I_{k}$. It implies that $c\in\sqrt{I_{k}} = Q = \sqrt
{I}.$
\end{proof}

\begin{proposition}
Let $I$ be a weakly 1-absorbing primary ideal of $R$ and $c$ be a nonunit
element of $R\backslash I$. Then $(I:c)$ is a weakly primary ideal of $R.$
\end{proposition}

\begin{proof}
Suppose that $0\neq ab\in(I:c)$ for some nonunit $c\in R\backslash I$ and
assume that $a\notin(I:c).$ Hence $b$ is a nonunit element of $R $. If $a$ is
unit, then $b\in(I:c)\subseteq\sqrt{(I:c)}$ and we are done. So assume that
$a$ is a nonunit element of $R$. Since $0\not = abc = acb \in I$ and $ac\notin
I$ and $I$ is a weakly 1-absorbing primary ideal of $R$, we conclude that
$b\in\sqrt{I}\subseteq\sqrt{(I:c)}$. Thus $(I:c)$ is a weakly primary ideal of
$R.$
\end{proof}

The next theorem gives a characterization for weakly 1-absorbing primary
ideals of $R=R_{1}\times R_{2}$ where $R_{1}$ and $R_{2}$ are commutative
rings with identity that are not fields.

\begin{theorem}
\label{w1}Let $R_{1}$ and $R_{2}$ be commutative rings with identity that are
not fields, $R=R_{1}\times R_{2}$, and $I$ be a a nonzero proper ideal of $R$.
Then the following statements are equivalent.

\begin{enumerate}
\item $I$ is a weakly 1-absorbing primary ideal of $R$.

\item $I$ $=I_{1}\times R_{2}$ for some primary ideal $I_{1}$ of $R_{1}$ or
$I=R_{1}\times I_{2}$ for some primary ideal $I_{2}$ of $R_{2}.$

\item $I$ is a 1-absorbing primary ideal of $R$.

\item $I$ is a primary ideal of $R_{1}$.
\end{enumerate}
\end{theorem}

\begin{proof}
(1)$\Rightarrow$(2). Suppose that $I$ is a weakly 1-absorbing primary ideal of
$R$. Then $I$ is of the form $I_{1}\times I_{2}$ for some ideals $I_{1}$ and
$I_{2}$ of $R_{1}$ and $R_{2}$, respectively. Assume that both $I_{1}$ and
$I_{2}$ are proper. Since $I$ is a nonzero ideal of $R$, we conclude that
$I_{1} \not = 0$ or $I_{2} \not = 0$. We may assume that $I_{1} \not = 0$. Let
$0\neq$ $c\in I_{1}$. Then $0\neq(1,0)(1,0)(c,1)=(c,0)\in I_{1}\times I_{2}.$
It implies that $(1,0)(1,0)\in I_{1}\times I_{2}$ or $(c,1)\in\sqrt
{I_{1}\times I_{2}}=\sqrt{I_{1}}\times\sqrt{I_{2}}$, that is $I_{1}=R_{1}$ or
$I_{2}=R_{2}$, a contradiction. Thus either $I_{1}$ or $I_{2}$ is a proper
ideal. Without loss of generality, assume that $I=I_{1}\times R_{2}$ for some
proper ideal $I_{1}$ of $R_{1}$. We show that $I_{1}$ is a primary ideal of
$R_{1}$. Let $ab\in I_{1}$ for some $a,b\in R_{1}$. We can assume that $a$ and
$b$ are nonunit elements of $R_{1}$. Since $R_{2}$ is not a field, there
exists a nonunit nonzero element $x\in R_{2}$. Then $0\neq(a,1)(1,x)(b,1)\in
I_{1}\times R_{2}$ which implies that either $(a,1)(1,x)\in I_{1}\times R_{2}$
or $(b,1)\in\sqrt{I_{1}\times R_{2}}=\sqrt{I_{1}}\times R_{2};$ i.e, $a\in
I_{1}$ or $b\in\sqrt{I_{1}}.$

(2)$\Rightarrow$(3). Since $I$ is a primary ideal of $R$, $I$ is a 1-absorbing
primary ideal of $R$ by \cite[Theorem 1(1)]{BadEce}.

(3)$\Rightarrow$(4). Since $I$ a 1-absorbing primary ideal of $R$ and $R$ is
not a quasilocal ring, we conclude that $I$ is a primary ideal of $R$ by
\cite[Theorem 3]{BadEce}.

(4)$\Rightarrow$(1). It is clear.
\end{proof}

\begin{theorem}
\label{fi} Let $R_{1},..., R_{n}$ be commutative rings with $1\neq0$ for some
$2 \leq n < \infty$, and let $R=R_{1}\times\cdots\times R_{n}$. Then the
following statements are equivalent.

\begin{enumerate}
\item Every proper ideal of $R$ is a weakly 1-absorbing primary ideal of $R$.

\item $n = 2$ and $R_{1}, R_{2}$ are fields.
\end{enumerate}
\end{theorem}

\begin{proof}
(1)$\rightarrow$(2). Suppose that every proper ideal of $R$ is a weakly
1-absorbing primary ideal. Without loss of generality, we may assume that $n =
3$. Then $I=R_{1}\times\{0\}\times\{0\}$ is a weakly 1-absorbing primary ideal
of $R.$ However, for a nonzero $a \in R_{1}$, we have $(0,0,0)\neq
(1,0,1)(1,0,1)(a,1,0) = (a,0,0)\in I$, but neither $(1,0,1)(1,0,1)\in I$ nor
$(a,1,0)\in\sqrt{I}$, a contradiction. Thus $n = 2$. Assume that $R_{1}$ is
not a field. Then there exists a nonzero proper ideal $A$ of $R_{1}$. Hence $I
= A\times\{0\}$ is a weakly 1-absorbing primary ideal of $R.$ However, for a
nonzero $a \in A$, we have $(0,0)\neq(1,0)(1,0)(a,1) = (a,0)\in I$, but
neither $(1,0)(1,0)\in I$ nor $(a,1) \in\sqrt{I}$, a contradiction. Similarly,
one can easily show that $R_{2}$ is a field. Hence $n = 2$ and $R_{1}, R_{2}$
are fields.

(2)$\rightarrow$(1). Suppose that $n = 2$ and $R_{1}, R_{2}$ are fields. Then
$R$ has exactly three proper ideals, i.e., $\{(0, 0)\}$, $\{0\}\times R_{2}$
and $R_{1}\times\{0\}$ are the only proper ideals of $R$. Hence it is clear
that each proper ideal of $R$ is a weakly 1-absorbing primary ideal of $R$.
\end{proof}

Since every ring that is a product of a finite number of fields is a
von-Neumann regular ring, in light of Theorem \ref{VNI} and Theorem \ref{fi}
we have the following result.

\begin{corollary}
\label{VNfi} Let $R_{1},..., R_{n}$ be commutative rings with $1\neq0$ for
some $2 \leq n < \infty$, and let $R=R_{1}\times\cdots\times R_{n}$. Then the
following statements are equivalent.

\begin{enumerate}
\item Every proper ideal of $R$ is a weakly 1-absorbing primary ideal of $R$.

\item Every proper ideal of $R$ is a weakly primary ideal of $R$.

\item $n = 2$ and $R_{1}, R_{2}$ are fields, and hence $R = R_{1} \times
R_{2}$ is a von-Neumann regular ring.
\end{enumerate}
\end{corollary}

\begin{theorem}
\label{f}Let $R_{1}$ and $R_{2}$ be commutative rings and $f:R_{1}\rightarrow
R_{2}$ be a ring homomorphism with $f(1)=1$. Then the following statements hold:
\end{theorem}

\begin{enumerate}
\item Suppose that $f$ is a monomorphism and $f(a)$ is a nonunit element of
$R_{2}$ for every nonunit element $a \in R_{1}$ (for example if $U(R_{2})$ is
a torsion group) and $J$ is a weakly 1-absorbing primary ideal of $R_{2}.$
Then $f^{-1}(J)$ is a weakly 1-absorbing primary ideal of $R_{1}.$

\item If $f$ is an epimorphism and $I$ is a weakly 1-absorbing primary ideal
of $R_{1}$ such that $Ker(f) \subseteq I$, then $f(I)$ is a weakly 1-absorbing
primary ideal of $R_{2}.$
\end{enumerate}

\begin{proof}
(1) Let $0\neq abc\in f^{-1}(J)$ for some nonunit elements $a,b,c\in R$. Since
$Ker(f)=0,$ we have $0\neq f(abc)=f(a)f(b)f(c)\in J$, where $f(a), f(b), f(c)$
are nonunit elements of $R_{2}$ by hypothesis. Hence $f(a)f(b)\in J $ or
$f(c)\in\sqrt{J}$. Hence $ab\in f^{-1}(J)$ or $c\in\sqrt{f^{-1}(J)}%
=f^{-1}(\sqrt{J}).$ Thus $f^{-1}(J)$ is a weakly 1-absorbing primary ideal of
$R_{1}.$

(2) Let $0\neq xyz\in f(I)$ for some nonunit elements $x,y,z\in R.$ Since $f$
is onto, there exists nonunit elements $a,b,c\in I$ such that $x=f(a),$
$y=f(b),$ $z=f(c)$. Then $f(abc)=f(a)f(b)f(c)=xyz\in f(I)$. Since $Ker(f)
\subseteq I,$ we have $0\neq abc\in I$. It follows $ab\in I$ or $c\in\sqrt{I
}$. Thus $xy\in f(I)$ or $z\in f(\sqrt{I})$. Since $f$ is onto and $Ker(f)
\subseteq I$, we have $f(\sqrt{I})$ $=\sqrt{f(I)}$. Thus we are done.
\end{proof}

The following example shows that the hypothesis in Theorem \ref{f}(1) is crucial.

\begin{example}
\label{CF} \cite[Example 1]{BadEce} Let $A=K[x,y]$, where $K$ is a field, $M =
(x, y)A$, and $B = A_{M}$. Note that $B$ is a quasilocal ring with maximal
ideal $M_{M}$. Then $I = xM_{M} = (x^{2}, xy)B$ is a 1-absorbing primary ideal
of $B$ (see \cite[Theorem 5]{BadEce}) and $\sqrt{I} = xB$. However $xy \in I$,
but neither $x \in I$ nor $y \in\sqrt{I}$. Thus $I$ is not a primary ideal of
$B$. Let $f : B\times B \rightarrow B$ such that $f(x, y) = x$. Then $f$ is a
ring homomorphism from $B\times B$ onto $B$ such that $f(1,1) = 1$. However,
$(1, 0)$ is a nonunit element of $B\times B$ and $f(1,0) = 1$ is a unit of
$B$. Thus $f$ does not satisfy the hypothesis of Theorem \ref{f}(1). Now
$f^{-1}(I) = I \times B$ is not a weakly 1-absorbing ideal of $B\times B$ by
Theorem \ref{w1}.
\end{example}

\begin{theorem}
Let $I$ be a proper ideal of $R.$ Then the following statements hold.

\begin{enumerate}
\item If $J$ is a proper ideal of a ring $R$ with $J\subseteq I$ and $I$ is a
weakly 1-absorbing primary ideal of $R$, then $I/J$ is a weakly 1-absorbing
primary ideal of $R/J$.

\item If $J$ is a proper ideal of a ring $R$ with $J\subseteq I$ such that
$U(R/J) = \{a + J \mid a \in U(R)\}$. If $J$ is a 1-absorbing primary ideal of
$R$ and $I/J$ is a weakly 1-absorbing primary ideal of $R/J$, then $I$ is a
1-absorbing primary ideal of $R$.

\item If $\{0\}$ is a 1-absorbing primary ideal of $R$ and $I$ is a weakly
1-absorbing primary ideal of $R$, then $I$ is a 1-absorbing primary ideal of
$R$.

\item If $J$ is a proper ideal of a ring $R$ with $J\subseteq I$ such that
$U(R/J) = \{a + J \mid a \in U(R)\}$. If $J$ is a weakly 1-absorbing primary
ideal of $R$ and $I/J$ is a weakly 1-absorbing primary ideal of $R/J$, then
$I$ is a weakly 1-absorbing primary ideal of $R$.
\end{enumerate}
\end{theorem}

\begin{proof}
(1) Consider the natural epimorphism $\pi:R\rightarrow R/J$. Then
$\pi(I)=I/J.$ So we are done by Theorem \ref{f} (2).

(2) Suppose that $abc\in I$ for some nonunit elements $a,b,c\in R$. If $abc\in
J$, then $ab\in J\subseteq I$ or $c\in\sqrt{J}\subseteq\sqrt{I}$ as $J$ is a
1-absorbing primary ideal of $R$. Now assume that $abc\notin J$. Then
$J\neq(a+J)(b+J)(c+J)\in$ $I/J$, where $a + J, b + J, c + J$ are nonunit
elements of $R/J$ by hypothesis. Thus $(a+J)(b+J)\in I/J$ or $(c+J)\in
\sqrt{I/J}.$ Hence $ab\in I$ or $c\in\sqrt{I}.$

(3) The proof follows from (2).

(4) Suppose that $0 \not = abc\in I$ for some nonunit elements $a,b,c\in R$.
If $abc\in J$, then $ab\in J\subseteq I$ or $c\in\sqrt{J}\subseteq\sqrt{I}$ as
$J$ is a weakly 1-absorbing primary ideal of $R$. Now assume that $abc\notin
J$. Then $J\neq(a+J)(b+J)(c+J)\in$ $I/J$, where $a + J, b + J, c + J$ are
nonunit elements of $R/J$ by hypothesis. Thus $(a+J)(b+J)\in I/J$ or
$(c+J)\in\sqrt{I/J}.$ Hence $ab\in I$ or $c\in\sqrt{I}.$
\end{proof}

In the following remark, we give the correct version of \cite[Theorem 17(1),
Corollary 3 and Corollary 4]{BadEce}.

\begin{remark}
\label{TT} Mohammed Tamekkante pointed out to the first-named author that in
\cite{BadEce}, we overlooked the fact that if $f: R_{1}\rightarrow R_{2}$ is a
ring homomorphism such that $f(1) = 1$, then it is possible that $f(a) \in
U(R_{2})$ for some nonunit element $a \in R_{1}$. Overlooking this fact caused
a problem in the proof of \cite[Theorem 17(1), Corollary 3 and Corollary
4]{BadEce}. We state the correct version of \cite[Theorem 17(1), Corollary 3
and Corollary 4]{BadEce}.

\begin{enumerate}
\item (\cite[Theorem 17(1)]{BadEce}). Let $R_{1}$ and $R_{2}$ be commutative
rings and $f:R_{1}\rightarrow R_{2}$ be a ring homomorphism with $f(1)=1$ such
that if $R_{2}$ is not a quasilocal ring, then $f(a)$ is a nonunit element of
$R_{2}$ for every nonunit element $a \in R_{1}$ (for example if $U(R_{2})$ is
a torsion group) and $J$ is a 1-absorbing primary ideal of $R_{2}. $ Then
$f^{-1}(J)$ is a 1-absorbing primary ideal of $R_{1}.$ (note that if $R_{2}$
is not a quasilocal ring, then $J$ is primary by \cite[Theorem 3]{BadEce}).

\item (\cite[Corollary 3]{BadEce}). Let $I$ and $J$ be proper ideals of a ring
R with $I \subseteq J$. If $J$ is a 1-absorbing primary ideal of $R$, then
$J/I$ is a 1-absorbing primary ideal of $R/I$. Furthermore, assume that if
$R/I$ is a quasilocal ring, then $U(R/I) = \{a + I \mid a \in U(R)\}$. If
$J/I$ is a 1-absorbing primary ideal of $R/I$, then $J$ is a 1-absorbing
primary ideal of $R$.

\item (\cite[Corollary 4]{BadEce}). Let $R$ be a ring and $A = R[x]$. Then a
proper ideal $I$ of $R$ is a 1-absorbing primary ideal of $R$ if and only if
$(I[x] + xA)/xA$ is a 1-absorbing primary ideal of $A/xA$ (The claim is clear
since $R$ is ring-isomorphic to $A/xA$)

Note that Example \ref{CF} shows that the hypothesis in (1) is crucial.
\end{enumerate}
\end{remark}

\begin{theorem}
\label{S}Let $S$ be a multiplicatively closed subset of $R,$ and $I$ a proper
ideal of $R$.\ Then the following statements hold.
\end{theorem}

\begin{enumerate}
\item If $I$ is a weakly 1-absorbing primary ideal of $R$ such that $I\cap
S=\emptyset$, then $S^{-1}I\ $is a weakly 1-absorbing primary ideal of
$S^{-1}R.$

\item If $S^{-1}I$ is a weakly 1-absorbing primary ideal of $S^{-1}R$ such
that $S\cap Z(R)=$ $\emptyset$ and $S\cap Z_{I}(R)=\emptyset$, then $I\ $is a
weakly 1-absorbing primary ideal of $R.$
\end{enumerate}

\begin{proof}
(1) Suppose that $0\neq\frac{a}{s_{1}}\frac{b}{s_{2}}\frac{c}{s_{3}}\in
S^{-1}I$ for some nonunit $a,b,c\in R\setminus S,$ $s_{1},s_{2},s_{3}\in S$
and $\frac{a}{s_{1}}\frac{b}{s_{2}}\notin S^{-1}I$. Then $0\neq uabc\in I$ for
some $u\in S$. Since $I$ is weakly 1-absorbing primary and $uab\notin I$, we
conclude $c\in\sqrt{I}$. Thus $\frac{c}{s_{3}}\in S^{-1}\sqrt{I}=\sqrt
{S^{-1}I}$. Thus $S^{-1}I\ $is a weakly 1-absorbing primary ideal of $S^{-1}R.
$

(2) Suppose that $0\neq abc\in I$ for some nonunit elements $a,b,c\in R$.
Hence $0\neq\frac{abc}{1}=\frac{a}{1}\frac{b}{1}\frac{c}{1}\in S^{-1}I$ as
$S\cap Z(R)=$ $\emptyset$ . Since $S^{-1}I$ is weakly 1-absorbing primary, we
have either $\frac{a}{1}\frac{b}{1}\in S^{-1}I$ or $\frac{c}{1}\in\sqrt
{S^{-1}I}=S^{-1}\sqrt{I}$. If $\frac{a}{1}\frac{b}{1}\in S^{-1}I,$ then
$uab\in I$ for some $u\in S.$ Since $S\cap Z_{I}(R)=\emptyset$, we conclude
that $ab\in I.$ If $\frac{c}{1}\in S^{-1}\sqrt{I},$ then $(tc)^{n}\in I$ for
some positive integer $n\geq1$ and $t\in S$. Since $t^{n}\notin Z_{I}(R)$, we
have $c^{n}\in I,$ i.e., $c\in\sqrt{I}$. Thus $I\ $is a weakly 1-absorbing
primary ideal of $R.$
\end{proof}

\begin{definition}
Let $I$ be a weakly 1-absorbing primary ideal of $R$ and $I_{1}I_{2}%
I_{3}\subseteq I$ for some proper ideals $I_{1},I_{2},I_{3}$ of $R$. If
$(a,b,c)$ is not 1-triple zero of $I$ for every $a\in I_{1},$ $b\in I_{2},$
$c\in I_{3}$, then we call $I$ a free 1-triple zero with respect to
$I_{1}I_{2}I_{3}.$
\end{definition}

\begin{theorem}
\label{tri}Let $I$ be a weakly 1-absorbing primary ideal of $R$ and $J$ be a
proper ideal of $R$ with $abJ\subseteq I$ for some $a,b\in R$. If $(a,b,j)$ is
not a 1-triple zero of $I$ for all $j\in J$ and $ab\notin I,$ then
$J\subseteq\sqrt{I}.$
\end{theorem}

\begin{proof}
Suppose that $J\nsubseteq\sqrt{I}.$ Then there exists $c\in J$ $\backslash
\sqrt{I}$. Then $abc\in abJ\subseteq I$. If $abc\neq0$, then it contradicts
our assumption that $ab\notin I$ and $c\notin\sqrt{I}.$ Thus $abc=0$. Since
$(a,b,c)$ is not a 1-triple zero of $I$ and $ab\notin I$, we conclude
$c\in\sqrt{I},$ a contradiction. Thus $J\subseteq\sqrt{I}.$
\end{proof}

\begin{theorem}
\label{free}Let $I$ be a weakly 1-absorbing primary ideal of $R$ and $0\neq
I_{1}I_{2}I_{3}\subseteq I$ for some proper ideals $I_{1},I_{2},I_{3}$ of $R$.
If $I$ is free 1-triple zero with respect to $I_{1}I_{2}I_{3}$, then
$I_{1}I_{2}\subseteq I$ or $I_{3}\subseteq\sqrt{I}.$
\end{theorem}

\begin{proof}
Suppose that $I$ is free 1-triple zero with respect to $I_{1}I_{2}I_{3}$, and
$0\neq I_{1}I_{2}I_{3}\subseteq I$. Assume that $I_{1}I_{2}\nsubseteq I.$ Then
there exist $a\in I_{1}$, $b\in I_{2}$ such that $ab\notin I$. Since $I$ is a
free 1-triple zero with respect to $I_{1}I_{2}I_{3},$ we conclude that
$(a,b,c)$ is not a 1-triple zero of $I$ for all $c\in I_{3}$. Thus
$I_{3}\subseteq\sqrt{I}$ by Theorem \ref{tri}.
\end{proof}

\textbf{Acknowledgement}. We would like to thank Mohammed Tamekkante for
pointing out a problem in \cite[Theorem 17(1)]{BadEce}.


\begin{thebibliography}{99}                                                                                               %
\bibitem {AB}D. D. Anderson and M. Batanieh, Generalizations of prime ideals,
Comm. Algebra, \textbf{36} (2008), 686--696.

\bibitem {And}D.D. Anderson and E. Smith, Weakly prime ideals, Houston Jornal
of Mathematics, \textbf{29} (4) (2003), 831-840.

\bibitem {AndBad}D.F. Anderson and A. Badawi, On $n$-absorbing ideals of
commutative rings, Comm. Algebra, \textbf{39} (2011), 1646-1672.

\bibitem {AndBad2}D.F. Anderson and A. Badawi, Von Neumann regular and related
elements in commutative rings, Algebra Colloquium, \textbf{19}(SPL. ISS. 1)
(2012), 1017--1040.

\bibitem {At}S. E. Atani and F. Farzalipour, On weakly primary ideals,
Georgian Mathematical Journal, 12 (3) (2005), 423-429.

\bibitem {Badawi}A.Badawi, On 2-absorbing ideals of commutative rings, Bull.
Austral. Math. Soc., \textbf{75} (2007), 417-429.

\bibitem {BadDar}A. Badawi and A. Y. Darani, On weakly 2-absorbing ideals of
commutative rings, Houston J. Math. 39 (2013), no. 2, 441--452.

\bibitem {Badawi2}A. Badawi, \"{U}. Tekir and E. Yetkin, On $2$-absorbing
primary ideals in commutative rings, Bull. Korean Math. Soc., \textbf{51} (4)
(2014), 1163-1173.

\bibitem {Badawi3}A. Badawi, \"{U}. Tekir and E. Yetkin, On weakly
$2$-absorbing primary ideals in commutative rings, Journal of Korean
Mathematical Society, \textbf{52} (1) (2015), 97-111.

\bibitem {BadEce}A.Badawi and E. Yetkin \c{C}elikel, On 1-absorbing primary
ideals of a commutative ring, Journal of Algebra and Its Applications (in press)

\bibitem {BadDenis}A. Badawi, D. Sonmez and G. Yesilot, On Weakly $\delta
$-Semiprimary Ideals of Commutative Rings, Algebra Colloquium, \textbf{25}(3)
(2018), 387-398.

\bibitem {GilmerR}R. W. Gilmer, Rings in which semiprimary ideals are primary,
Pacific Journal of Math. \textbf{12(4)} (1962) 1273-1276.

\bibitem {Gilmer}R. Gilmer, Multiplicative ideal theory, Queen Papers Pure
Appl. Math. 90, Queen's University, Kingston, 1992.

\bibitem {Huc}J. Huckaba, Rings with zero-divisors, Marcel Dekker, NewYork/ Basil,1988.

\bibitem {Kap}I. Kaplansky, Commutative rings, rev. ed., University of Chicago
Press, Chicago, 1974.

\bibitem {LM}M.D. Larson and P.J. McCarthy, Multiplicative theory of ideals,
Academic Press, New York, London, 1971.

\bibitem {Qu}P. Quartararo and H. S. Butts, Finite unions of ideals and
modules, Proc. Amer. Math. Soc., \textbf{52} (1975), 91-96.

\bibitem {OP}O. Zariski and Pierre Samuel, Commutative Algebra. V.I.
(Princeton, 1958).
\end{thebibliography}
\end{document}